\documentclass[11pt,twoside]{amsart}

\usepackage{latexsym}
\usepackage{amsmath}
\usepackage{amssymb}
\usepackage{amsthm}
\usepackage{amscd}

\usepackage{enumerate} 
\usepackage{amssymb} 
\usepackage{mathrsfs}          
\usepackage{multicol}

\def\ZZ{{\mathbb{Z}}}
\def\QQ{{\mathbb{Q}}}
\def\R{{\mathscr{R}}}
\def\C{{\mathscr{C}}}
\def\CC{{\mathbb{C}}}
\def\T{{\mathscr{T}}}
\def\K{{\mathscr{K}}}

\def\PP{{\mathbb{P}}}
\def\LL{{\mathbb{L}}}
\def\AA{{\mathbb{A}}}
\def\H{{\mathscr{H}}}

\usepackage{latexsym}

\newtheorem{them}{Theorem}[section]

\newtheorem{pro}[them]{Proposition}

\newtheorem{exa}[them]{Example}

\newtheorem{lem}[them]{Lemma}

\newtheorem{rem}[them]{Remark}

\newtheorem{cor}[them]{Corollary}

\newtheorem{defi}[them]{Definition}

\newtheorem{conj}[them]{Conjecture}

\newtheorem{nota}[them]{Notation}

\newtheorem{cl}[them]{Claim}

\title{Lengths of chains of minimal rational curves on Fano manifolds}

\author{Kiwamu Watanabe}

\date{30 June 2009. Revised: 16 September 2010.}

\address{Department of Mathematical Sciences  School of Science and Engineering Waseda University, 
4-1 Ohkubo 3-chome 
Shinjuku-ku 
Tokyo 169-8555 
Japan}
\email{kiwamu0219@fuji.waseda.jp}


\begin{document}

\maketitle


\begin{abstract}
In this paper, we consider a natural question how many minimal rational curves are needed to join two general points on a Fano manifold $X$ of Picard number $1$. In particular, we study the minimal length of such chains in the cases where the dimension of $X$ is at most $5$, the coindex of $X$ is at most $3$ and $X$ equips with a structure of a double cover. As an application, we give a better bound on the degree of Fano $5$-folds of Picard number $1$. 
\end{abstract}

\section{Introduction}  

For a Fano manifold, a {\it minimal rational component} $\K$ is defined to be a dominating irreducible component of the normalization of the parameter space of rational curves whose degree is minimal among such components and a {\it variety of minimal rational tangents} is the parameter space of the tangent directions of $\K$-curves at a general point. Nowadays these two objects often appear in the study of Fano manifolds \cite{Hw2,KeS}. On the other hand, chains of rational curves also play an important role in the field. For instance, Koll\'ar-Miyaoka-Mori \cite{KMM} and Nadel \cite{Na} independently showed the boundedness of the degree of Fano manifolds of Picard number $\rho=1$ by using chains of rational curves. 
From these viewpoints, it is a natural question how many rational curves in the family $\K$ are needed to join two general points. We denote by $l_{\K}$ the minimal length of such chains of general $\K$-curves. 
In this direction, Hwang and Kebekus \cite{HK} developed an infinitesimal method to study the lengths of Fano manifolds via the varieties of minimal rational tangents. They also dealt with some examples when the varieties of minimal rational tangents and those secant varieties are simple, such as complete intersections, Hermitian symmetric spaces and homogeneous contact manifolds. Furthermore the following was obtained.

\begin{them}[\cite{HK,IR}]\label{bound} Let $X$ be a prime Fano $n$-fold of $\rho=1$. If the Fano index $i_X$ satisfies $n+1>i_X > \frac{2}{3}n$, then $l_{\K}=2$.
\end{them}

A Fano manifold is $\it{prime}$ if the ample generator of the Picard group is very ample. Our original motivation of this paper is to compute the lengths of Fano manifolds of coindex $\leq 3$. By the above theorem, it is sufficient to consider the cases where $n \leq 5$, $(n,i_X)=(6,4)$ and $X$ is non-prime. Remark that non-prime Fano manifolds of coindex $\leq 3$ admit double cover structures \cite{F1,F2,F3,Muk,Me}. 
First we show the following by using the method of Hwang and Kebekus (Precise definitions of notations are given in Section~$2$ and $4$.): 

\begin{them}[Theorem~\ref{length2}, Theorem~\ref{length3}]\label{MT1} Let $X$ be a Fano $n$-fold of $\rho=1$, $\K$ a minimal rational component of $X$ and $p+2$ the anti-canonical degree of rational curves in $\K$. Then if $p=n-3>0$, we have $l_{\K}=2$ and if $(n,p)=(5,1)$, we have $l_{\K}=3$.

\end{them}

By combining this theorem and well-known or easy arguments, we obtain the following table (see Remark~\ref{3fold}, Theorem~\ref{4fold} and Theorem~\ref{5fold}). In particular, when $n \leq 5$, $l_{\K}$ depends only on $(n,p)$.

\begin{center}
\begin{tabular}{|c|c|c||c|c|c||c|c|c|}
\hline
 $n$ & $p$ & $l_{\K}$ & $n$ & $p$ & $l_{\K}$ & $n$ & $p$ & $l_{\K}$ \\ \hline \hline
 $3$ & $2$ & $1$ &      $4$ & $3$ &  $1$ & $5$ & $4$ & $1$ \\
 $3$ & $1$ & $2$ &      $4$ & $2$ &  $2$ & $5$ & $3$ & $2$ \\
 $3$ & $0$ & $3$ &      $4$ & $1$ &  $2$ & $5$ & $2$ & $2$ \\
 $$  & $$  &  $$ &      $4$ & $0$ &  $4$ & $5$ & $1$ & $3$  \\
 $$  & $$  & $ $ &      $$  & $$  &  $$  & $5$ & $0$ & $5$ \\
 
   \hline
\end{tabular}

\end{center}

As a corollary, we get a better bound on the degree of Fano $5$-folds of $\rho=1$.

\begin{cor}[Corollary~\ref{canobound}] For a Fano $5$-fold of $\rho=1$, $(-K_X)^5 \leq 9^5=59049$.\end{cor} 

On the other hand, the following shows $l_{\K}$ does not depend only on $(n,p)$ in general.

\begin{them}[Theorem~\ref{coindex3}]\label{MT2} Let $X$ be a Fano manifold of $\rho=1$ with coindex $3$ and $\K$ a minimal rational component of $X$. Assume that $n:= \dim X \geq 6$. Then $l_{\K}=2$ except the case $X$ is a $6$-dimensional Lagrangian Grassmannian $LG(3,6)$. In the case $X=LG(3,6)$, we have $l_{\K}=3$.
\end{them}

Consequently, we obtain the following table ($n \geq 6$).

\begin{center}
\begin{tabular}{|c|c|c|}
\hline
 $X$ & $i_X$ & $l_{\K}$  \\ \hline \hline
 $\PP^n$ & $n+1$ & $1$    \\
 $\QQ^n$ & $n$  & $2$  \\
 del Pezzo mfd. of dim. $n$ & $n-1$ & $2$  \\
 Mukai mfd. of dim. $n \geq 7$ & $n-2$ & $2$ \\
 Mukai mfd. of dim. $6$ & $4$ &  $2$ or $3$ \\
   \hline
\end{tabular}

\end{center}

In Theorem~\ref{2/3}, we give a classification of prime Fano $n$-folds satisfying $i_X=\frac{2}{3}n$ and $l_{\K} \neq 2$. These are extremal cases of Theorem~\ref{bound}. Except the case $n=3$, these varieties are deeply related to Severi varieties which are classified by Zak \cite{Za} (see Corollary~\ref{Zak1}). Furthermore, for prime Fano manifolds, we discuss a relation among {\it $2$-connectedness by lines, conic-connectedness} and {\it defectiveness of the secant varieties}. 
In the last section, we investigate Fano manifolds which equip with structures of double covers and are covered by rational curves of degree $1$, by a geometric argument without using varieties of minimal rational tangents. In Proposition~\ref{criterion}, we give a criterion for such Fano manifolds to be $2$-connected. Remark that all Fano manifolds dealt in \cite{HK} as examples are prime. However our cases include some non-prime Fano manifolds. Throughout this paper, we work over the complex number field $\CC$.

\section{Deformation theory of rational curves and varieties of minimal rational tangents}

First we review some basic facts of deformation theory of rational curves and the definition of varieties of minimal rational tangents. 
For detail, we refer to \cite{Hw2,Ko} and follow the conventions of them. 

Throughout this paper, unless otherwise noted, we always assume that $X$ is a Fano manifold of ${\rm Pic} (X) \cong \ZZ[{\mathscr{O}}_X(1)]$, where ${\mathscr{O}}_X(1)$ is the ample generator, and denote by ${\rm RatCurves}^n(X)$ the normalization of the space of integral rational curves on $X$. We also assume $n:=\dim X \geq 3$. We denote by $i_X$ the {\it Fano index of $X$} which is the integer satisfying $\omega_X \cong {\mathscr{O}}_X(-i_X)$, where $\omega_X = {\mathscr{O}}_X(K_X)$ is the canonical line bundle of $X$. We call $n+1-i_X$ the {\it coindex of $X$}.

As is well known, a Fano manifold is uniruled. It is equivalent to the condition that there exists a free rational curve $f: \PP^1 \rightarrow X$. Here we call a rational curve $f: \PP^1 \rightarrow X$ {\it free} if $f^*T_X$ is semipositive. 
An irreducible component $\K$ of ${\rm RatCurves}^n(X)$ is called a {\it minimal rational component} if it contains a free rational curve of minimal anti-canonical degree. We denote by ${\K}_x$ the normalization of the subscheme of $\K$ parametrizing rational curves passing through $x$. Since each member of $\K$ is numerically equivalent, we can define the ${\mathscr{O}}_X(1)$-degree of $\K$ which is denoted by $d_{\K}$. We will use the symbol $p$ to denote $i_Xd_{\K}-2$. In this setting, the minimal rational component $\K$ satisfies the following fundamental properties.

\begin{pro}[see \cite{Hw2}]\label{stand}
\begin{enumerate} 
\item{For a general point $x \in X$, ${\K}_x$ is a disjoint union of smooth projective varieties of dimension $p$.} 
\item{For a general member $[f]$ of $\K$, $f^*T_X \cong {\mathscr{O}}(2) \oplus {\mathscr{O}}(1)^p \oplus {\mathscr{O}}^{n-1-p}$ which is called a {\it standard rational curve}. In particular, $p \leq n-1$}.

\end{enumerate}
\end{pro}

\begin{them} [\cite{Ke2}] For a general point $x \in X$, there are only finitely many curves in $\K_x$ which are singular at $x$.
\end{them}

For a general point $x \in X$, we define the tangent map ${\tau}_x : {\K}_x \rightarrow \PP(T_xX)$\footnote{For a vector space $V$, $\PP(V)$ denotes the projective space of lines through the origin in $V$.} by assigning the tangent vector at $x$ to each member of $\K_x$ which is smooth at $x$. We denote by $\C_x \subset \PP(T_xX)$ the image of ${\tau}_x$, which is called the {\it variety of minimal rational tangents} at $x$.

\begin{them} [\cite{HM2,Ke2}] The tangent map ${\tau}_x : {\K}_x \rightarrow \C_x \subset \PP(T_xX)$ is the normalization.
\end{them}

\begin{them}[\cite{CMSB,Ke1}]\label{CMSB} If $p=n-1$, namely $\C_x =\PP(T_xX)$, then $X$ is isomorphic to $\PP^n$.
\end{them}

\begin{them}[\cite{HH}]\label{HHM} Let $S=G/P$ a rational homogeneous variety corresponding to a long simple root and $\C_o \subset \PP(T_oS)$ the variety of minimal rational tangents at a reference point $o \in S$. Assume $\C_o \subset \PP(T_oS)$ and $\C_x \subset \PP(T_xX)$ are isomorphic as projective subvarieties. Then $X$ is isomorphic to $S$.
\end{them}

\begin{them}[\cite{Mi}]\label{Mi} If $X$ is a Fano manifold of $n:=\dim X \geq 3$, the following are equivalent.
\begin{enumerate}
\item{$X$ is isomorphic to a smooth quadric hypersurface $\QQ^n$.}
\item{The Picard number of $X$ is $1$ and the minimal value of the anti-canonical degree of rational curves passing through a very general point $x_0 \in X$ is equal to $n$.}
\end{enumerate}
\end{them}

\begin{cor}\label{Mi2} If $p=n-2$, namely $\C_x \subset \PP(T_xX)$ is a hypersurface, $X$ is isomorphic to $\QQ^n$.
\end{cor}

\begin{proof} For a very general point $x_0 \in X$, any rational curve passing through $x_0$ is free. Let $C_0$ be a rational curve passing through $x_0$ whose degree is minimal among such rational curves and $\H \subset {\rm RatCurves}^n(X)$ an irreducible component containing $[C_0]$. Then $\H$ is a dominating family. It implies that the anticanonical degree of $\H$ is equal to one of $\K$. Furthermore the anticanonical degree of $\K$ is $n$ from our assumption. Therefore $X$ is isomorphic to $\QQ^n$ by Theorem~\ref{Mi}.

\end{proof}

\section{Varieties of minimal rational tangents in the cases $p=n-3$ and $(n, p)=(5, 1)$ }

\begin{pro}[{\cite[Proposition 1.4, Proposition 1.5, Theorem 2.5]{Hw2}, \cite[Proposition 2, Proposition 5]{Hw3}, \cite[Proposition~2.2]{Hw4}}]\label{hwlem1} Let $X$, $\K$ and $p$ be as in Section $2$ and $\C_x$ the variety of minimal rational tangents associated to $\K$ at a general point $x \in X$. 
\begin{enumerate}
\item{The tangent map $\tau_x:\K_x \rightarrow \C_x \subset \PP(T_xX)$ is an immersion at $[C] \in \K_x$ if $C$ is a standard rational curve on $X$.}
\item{If $X \subset \PP^N$ is covered by lines, the tangent map $\tau_x$ is an embedding. In particular, $\C_x$ is a disjoint union of smooth projective varieties of dimension $p$. }
\item{If $2p>n-3$ and $\C_x$ is smooth, $\C_x \subset \PP(T_xX)$ is non-degenerate.}
\item{If $\C_x$ is reducible, it has at least three components.}
\item{If $\C_x$ is a union of linear subspaces of dimension $p>0$, two distinct irreducible components of $\C_x$ are disjoint.}
\item{$\C_x$ cannot be an irreducible linear subspace.}

\end{enumerate}

\end{pro}

\begin{pro}[{\cite[Proposition 9]{HM1}}]\label{hmlem} Let $X$, $\K$ and $\C_x$ be as above, $\PP(W_x)$ the linear span of $\C_x$ and $\T_x \subset \PP(\wedge^2W_x)$ the subvariety parametrizing tangent lines of the smooth locus of $\C_x$. Then $\T_x$ is contained in $\PP({\rm Ker}([~,~]_x)) \subset \PP(\wedge^2W_x)$, where $[~,~]_x: \wedge^2W_x \rightarrow T_xX/W_x$ is the Frobenius bracket tensor.  

\end{pro}

\begin{lem}\label{tang} If $X \subset \PP(V)$ is an irreducible hypersurface which is not linear, then its variety of tangential lines $\T_X \subset \PP(\wedge^2V)$ is non-degenerate.
\end{lem}

\begin{proof} Assume that $\T_X \subset \PP(\wedge^2V)$ is degenerate. We denote by $C(X) \subset V$ the cone corresponding to $X \subset \PP(V)$. Then there exists a nonzero $\omega \in \wedge^2V^*$ such that $C(X)$ is isotropic with respect to $\omega$. We set $Q:=\{v \in V| \omega (v, w)=0$ for any $w \in V\}$. $\omega$ induces a nonzero symplectic form on V/Q. For the projection $\pi: V \rightarrow V/Q$, it follows $2 \dim \pi(C(X)) \leq \dim V/Q$. Remark that $\dim V -1 = \dim C(X)$. Therefore we have an inequality $\dim V/Q \leq 2$. Since $\pi (C(X))$ is not $V/Q$ and $\{0\}$, it implies that $\pi (C(X)) \subset V/Q$ is a line. Hence $C(X) \subset V$ is a hyperplane. This contradicts the non-linearity of $X$.

\end{proof}

\begin{pro}[{\cite[Proposition 2]{Hw1}}]\label{hwlem2} Let $X$, $\K$, $\C_x$, $W_x$ be as in Proposition~\ref{hmlem} and $W$ be the distribution defined by $W_x$ for general $x \in X$. Then $W$ is integrable if and only if $W_x$ coincides with $T_xX$ for general $x \in X$.

\end{pro}

\begin{pro}[cf. {\cite[Proposition~7]{Hw3}}]\label{n-3} Let $X$ be a Fano $n$-fold of $\rho=1$ and $\K$ a minimal rational component of $X$ with $p=n-3 > 0$. Then the variety of minimal rational tangents $\C_x \subset \PP(T_xX)$ at a general point $x \in X$ is one of the following:
\begin{enumerate}
\item {a non-degenerate variety with no linear component, or}
\item {a disjoint union of at least three lines. }
\end{enumerate}

\end{pro}

\begin{proof} First assume that $\C_x$ has a linear component. Then every component of $\C_x$ is linear. This follows from the irreducibility of $\C$, where $\C \subset \PP(T_X)$ is the closure of the union of the $\C_x$ for general points $x \in X$. 
By Proposition~\ref{hwlem1}, $\C_x$ is a disjoint union of at least three linear subspaces. Let $\C_{x,1}$ and $\C_{x,2}$ be distinct components of $\C_x$. Since $\dim \C_{x,1} + \dim \C_{x,2} - \dim \PP(T_xX)=n-5$, we have $\C_{x,1} \cap \C_{x,2} \neq \emptyset$ if $n \geq 5$. This implies $n=4$ and $\C_x$ is a disjoint union of at least three lines. 
 
Second assume that $\C_x$ has no linear components. We will show that $\C_x \subset \PP(T_xX)$ is non-degenerate. Suppose $\C_x \subset \PP(T_xX)$ is degenerate. Let $\PP(W_x)$ be the linear span of $\C_x$ and $\T_x \subset \PP(\wedge^2W_x)$ be the subvariety parametrizing tangent lines of the smooth locus of $\C_x$. By Proposition~\ref{hmlem}, we have $\T_x \subset \PP({\rm Ker}([~,~]_x)) \subset \PP(\wedge^2W_x)$, where $[~,~]_x: \wedge^2W_x \rightarrow T_xX/W_x$ is the Frobenius bracket tensor. Lemma~\ref{tang} implies that $\T_x \subset \PP(\wedge^2W_x)$ is non-degenerate. Therefore $\PP({\rm Ker}([~,~]_x))$ coincides with $\PP(\wedge^2W_x)$. Applying Frobenius Theorem, the distribution $W$ defined by $W_x$ is integrable. However, this contradicts Proposition~\ref{hwlem2}. 
\end{proof}

By the same argument, we can show the following:

\begin{pro}\label{5,1} If $(n, p)= (5, 1)$, then the variety of minimal rational tangents $\C_x \subset \PP(T_xX)$ at a general point $x \in X$ satisfies one of the following:
\begin{enumerate}
\item {a curve with no linear component whose linear span $<\C_x>$ has dimension at least $3$, or}
\item {a disjoint union of at least three lines. }
\end{enumerate}

\end{pro}

\section{Spanning dimensions of loci of chains}

For simplicity, throughout this section, we continue to work under the same assumption as in Section~$2$ except Definition~\ref{sv} and Remark~\ref{defective}, that is, $X$ is a Fano $n$-fold of $\rho=1$ with $n \geq 3$ and $\K$ is a minimal rational component of $X$.
Remark that we can also work in a slight more general situation (in the category of uniruled manifolds). 

\begin{defi}[\cite{HK}]\label{loc} \rm Let $X$ and $\K$ be as above. For a general point $x \in X$, we define 

${\rm loc}^1(x):= \displaystyle{\bigcup_{[C] \in \K_x} C}$ and ${\rm loc}^{k+1}(x):= \overline{ \displaystyle{\bigcup_{[C] \in \K_y ~for~general~y \in {\rm loc}^k(x)}}C}$ inductively.

We denote the maximal value of the dimensions of irreducible components of ${\rm loc}^k(x)$ by $d_k$.
\end{defi}

\begin{defi}[\cite{HK}]\label{leng} \rm If there exists an integer $l$ such that $d_l=\dim X$ but $d_{l-1} < \dim X$, we say that $X$ has {\it length $l$ with respect to $\K$}, or $X$ is {\it $l$-connected by $\K$}. We denote by $l_{\K}$ the length. 

\end{defi}

By our assumption that the Picard number of $X$ is $1$, we can define the length.

\begin{pro}[{\cite{KMM1}, \cite[Lemma 1.3]{KMM}, \cite{Na} and \cite[Corollary IV.4.14]{Ko}}]\label{Na} Let $X$ and $\K$ be as above. 
Then there exists $l_{\K}$, that is, two general points on $X$ can be connected by a finite number of rational curves in $\K$. Furthermore we have  $l_{\K} \leq \dim X$.

\end{pro}

\begin{defi}\label{sv} \rm For varieties $X, Y \subset \PP^N$, we define the {\it join of $X$ and $Y$} by the closure of the union of lines passing through distinct two points $x \in X$ and $y \in Y$ and denote by $S(X, Y)$. In the special case that $X=Y$, $S^1X:=S(X,X)$ is called the {\it secant variety of $X$}.
\end{defi}

\begin{rem}\label{defective} \rm In general, it is easy to see the dimension of the secant variety $S^1X$ is at most $2n +1$, where $n:=\dim X$. The expected dimension of the secant variety $S^1X$ is $2n +1$. When the dimension of $S^1X$ is less than $2n +1$, we say the secant variety $S^1X$ {\it defective}.
\end{rem}

In \cite{HK} Hwang and Kebekus computed the first spanning dimension $d_1$ and gave a lower bound of the second $d_2$ (resp. $d_k$) under the assumption $\K_x$ is irreducible for a general point $x \in X$ by using the secant variety of the variety of minimal rational tangents. However their proof works even if we drop the assumption on the irreducibility of $\K_x$.

\begin{them}[\cite{HK, KeS}]\label{HK} Let $X$ be a Fano $n$-fold of $\rho=1$ with $n \geq 3$ and $\K$ a minimal rational component of $X$. Then,  without the assumption that $\K_x$ is irreducible for a general point $x \in X$, we have
\begin{enumerate}
\item{$d_1=p+1$ and $d_k \leq k(p+1)$ for each $k \geq 1$,}
\item{$d_{2} \geq \dim S^1{\C}_x +1$, if $p>0$.}

\end{enumerate}
\end{them}

\begin{proof} The former follows from Mori's Bend and Break and an easy induction on $k$. For the later, there is a proof in \cite[Theorem~21]{KeS} which is easier than one in \cite{HK}.

\end{proof}

\begin{lem}\label{0}  Let $X$ and $\K$ be as above. If $p=0$, we have $l_{\K}=n$.
\end{lem}

\begin{proof} We have an inequality $d_{k+1} \leq d_k +1$. In particular, $d_k \leq k$. By combining Proposition~\ref{Na}, we obtain our assertion. 
\end{proof}

\begin{rem}\label{3fold} \rm If $X$ is a Fano $3$-fold of $\rho=1$ which is not isomorphic to $\PP^3$ and $\QQ^3$, then $l_{\K}= 3$. Hence in the $3$-dimensional case, we have the following table:

\begin{center}
\begin{tabular}{|c|c|c|c|c|}
\hline
 $p$ & $i_X$ & $X$ & $l_{\K}$ & $d_{\K}$ \\ \hline \hline
 $2$ & $4$ & $\PP^3$ & $1$ &  $1$   \\
 $1$ & $3$ & $\QQ^3$ & $2$ &  $1$ \\
 $0$ & $2$ & del Pezzo & $3$ &  $1$ \\
 $0$ & $1$ & Mukai  & $3$ &  $2$ \\
 
   \hline
\end{tabular}

\end{center}
\end{rem}

\section{Lengths of Fano manifolds of dimension $\leq 5$}

\begin{lem}\label{sec} Let $X, Y \subset \PP^{N}$ be irreducible projective varieties of dimension $n$. Then the following holds.
\begin{enumerate}
\item{${\rm Vert}(X):=\{ p \in X| S(p,X)=X \} \subset \PP^N $ is a linear subspace.}  
\item{If ${\rm codim}({\rm Vert}(X),X) \leq 1$, then ${\rm Vert}(X)=X$.}
\item{If $\dim S(X,Y)=n+1$, then $X, Y \subset {\rm Vert}(S(X,Y))$.}
\item{If $N=n+2$ and $X \subset \PP^{n+2}$ is a non-degenerate variety which is not linear, then $S^1X = \PP^{n+2}$.}
\item{If $X$ or $Y$ is non-linear and $\dim S(X,Y)=n+1$, then $S(X, Y) \subset \PP^N$ is a linear subspace.}

\end{enumerate}
\end{lem}

\begin{proof} $\rm (i)$ It is well known that the vertex ${\rm Vert}(X) \subset \PP^N$ is a linear subspace (see \cite[Proposition 4.6.2]{FOV}). 

$\rm (ii)$ Suppose that ${\rm codim}({\rm Vert}(X),X) \leq 1$ and ${\rm Vert}(X) \neq X$. Then $\dim{\rm Vert}(X) = n-1$. For a point $x \in X \setminus {\rm Vert}(X)$, it turns out that $X$ coincides with a linear space $S(x, {\rm Vert}(X))$. Hence ${\rm Vert}(X)$ is an $n$-dimensional linear space. This is a contradiction.

$\rm (iii)$ It is enough to prove that $Y \subset {\rm Vert}(S(X,Y))$. Assume that $\dim S(X,Y)=n+1$. If $Y$ is contained in ${\rm Vert}(X)$, we see $S(X,Y)=X$. This contradicts our assumption. So this implies that $Y$ is not contained in ${\rm Vert}(X)$. For $y \in Y \setminus {\rm Vert}(X)$, we obtain $S(y,X)=S(Y,X)$. Furthermore we see
\begin{equation}
S(y,S(X,Y))=S(y,S(y,X))=S(y,X)=S(Y,X).
\end{equation}
Thus our assertion holds.

$\rm (iv)$ Assume that $S^1X \neq \PP^{n+2}$. Then $X=S^1X$ or $\dim S^1X=n+1$. If $X=S^1X$, $X$ is linear. It gives a contradiction. So we have $\dim S^1X=n+1$. Then $\rm (iii)$ implies that $X \subset {\rm Vert}(S^1X) \subset S^1X$. Since $X \subset \PP^{n+2}$ is non-degenerate, we get a contradiction.

$\rm (v)$ Because $\dim S(X,Y)=n+1$, it follows from $\rm (iii)$ that $X, Y \subset {\rm Vert}(S(X,Y)) \subset S(X,Y)$. Then we have ${\rm Vert}(S(X,Y)) = S(X,Y)$ because al least one of $X$ and $Y$ is not linear. So $S(X,Y)$ is a linear space.

\end{proof}

\begin{them}\label{length2} Let $X$ be a Fano manifold of $\rho=1$ with $n=\dim X \geq 4$. Assume that $X$ has a minimal rational component $\K$ with $p=n-3 > 0$. Then $X$ is $2$-connected by $\K$. In particular, if the Fano index $i_X$ is $n-1$, then $X$ is $2$-connected by lines.
\end{them}

\begin{proof} By Proposition~\ref{n-3}, the variety of minimal rational tangents $\C_x \subset \PP(T_xX)$ is 
\begin{enumerate}
\item{a non-degenerate variety with no linear component, or} 
\item{a disjoint union of at least three lines.}
\end{enumerate}
Let $\C_x$ be as in $\rm (i)$. If $\C_x$ is irreducible, Lemma~\ref{sec} $\rm (iv)$ implies $S^1{\C}_x = \PP(T_xX)$. On the other hand, in the case where $\C_x$ is reducible, $S^1{\C}_x = \PP(T_xX)$ also holds. In fact, for the irreducible decomposition $\C_x = \C_{x,1} \cup \dots \cup \C_{x,m}$, we assume that $S(\C_{x,i}, \C_{x,j}) \neq \PP(T_xX)$ for any $i,j$. Then we see $\dim S(\C_{x,i}, \C_{x,j})=n-2$. Hence Lemma~\ref{sec} $\rm (v)$ implies that $S(\C_{x,i}, \C_{x,j})$ is a linear subspace $\PP^{n-2} \subset \PP(T_xX)$. It turns out from Proposition~\ref{hwlem1} that $m \geq 3$. From now on, we focus on $(i,j)=(1,2)$. Because $\C_x \subset \PP(T_xX)$ is non-degenerate, there exists $j$ such that $S(\C_{x,1}, \C_{x,2}) \neq S(\C_{x,1}, \C_{x,j})$. We may assume that $j$ is $3$. 
We have $\C_{x,1} \subset S(\C_{x,1}, \C_{x,2}) \cap S(\C_{x,1}, \C_{x,3})$. Furthermore since $S(\C_{x,1}, \C_{x,2})$ and $S(\C_{x,1}, \C_{x,3})$ are distinct linear subspaces of dimension $n-2$, these intersection is a linear subspace of dimension $n-3$. Thus we have $\C_{x,1} = S(\C_{x,1}, \C_{x,2}) \cap S(\C_{x,1}, \C_{x,3})$. However this contradicts our assumption that $\C_x$ has no linear components. If $\C_x$ is as in $\rm (ii)$, we also have $S^1{\C_x}= \PP(T_xX)$.

As a consequence, in any case we have $S^1{\C}_x = \PP(T_xX)$. This implies that $d_2=n$. On the other hand, since $d_1=p+1=n-2<n$, $X$ is $2$-connected by $\K$. If $i_X=n-1 \geq 3$, then it follows from the equation $p+2=i_Xd_{\K}$ that $p=n-3$. 
\end{proof}

\begin{rem} \rm If $X$ is a prime Fano manifold with $i_X=n-1$ which is a del Pezzo manifold whose degree is at least 3, then the latter statement of the above theorem follows from Theorem~\ref{bound}.
\end{rem}

\begin{them}\label{4fold} Let $X$ be a Fano $4$-fold of $\rho=1$. Then we have the following table:
\begin{center}
\begin{tabular}{|c|c|c|c|c|}
\hline
 $p$ & $i_X$ & $X$ & $l_{\K}$ & $d_{\K}$ \\ \hline \hline
 $3$ & $5$ & $\PP^4$ & $1$ &  $1$   \\
 $2$ & $4$ & $\QQ^4$ & $2$ &  $1$ \\
 $1$ & $3$ & del Pezzo & $2$ &  $1$ \\
 $1$ & $1$ & coindex $4$ & $2$ &  $3$ \\
 $0$ & $2$ & Mukai & $4$ &  $1$ \\
 $0$ & $1$ & coindex $4$ & $4$ &  $2$ \\

   \hline
\end{tabular}

\end{center}

\end{them}

\begin{proof} The computation of the length with respect to $\K$ is an immediate consequence of Theorem~\ref{CMSB}, Corollary~\ref{Mi2}, Lemma~\ref{0} and Theorem~\ref{length2}. The other parts follow from the relation $p+2=i_Xd_{\K}$.

\end{proof}

Here we introduce two fundamental lemmata.

\begin{lem}\label{sec2} For an irreducible non-degenerate projective curve $C \subset \PP^N$, $\dim S^1C= {\rm min}\{3,N\}$.
\end{lem}

\begin{proof} In general, we see that $\dim S^1C \leq 3$. So when $\dim S^1C=3$, there is nothing to prove. If $\dim S^1C=1$, then $C$ is a line. Hence our claim holds. Thus we assume that $\dim S^1C=2$. Then it follows from Lemma~\ref{sec} $\rm (v)$ that $S^1C$ is a plane. As a consequence, we see $\dim S^1C= {\rm min}\{3,N\}$.
\end{proof}

\begin{lem}[{\cite[Remark 4.6.10]{FOV}}]\label{sec3} For two distinct integral curves $C, D \subset \PP^N$, $\dim S(C,D)=2$ holds if and only if $C \cup D $ is a plane curve.
\end{lem}

\begin{proof} The "only if" part comes from Lemma~\ref{sec} $\rm (v)$. The converse is trivial.
\end{proof}

\begin{them}\label{length3} Let $X$ be a Fano $5$-fold of $\rho=1$. Assume that $X$ has a minimal rational component $\K$ with $p=1$. Then $X$ is $3$-connected by $\K$. 
\end{them}

\begin{proof} By the same argument as in Theorem~\ref{length2}, we can prove this theorem. In fact, Proposition~\ref{5,1}, Lemma~\ref{sec2} and Lemma~\ref{sec3} imply that $\dim S^1{\C}_x \geq 3$ for the variety of minimal rational tangents ${\C}_x \subset \PP(T_xX)$. It turns out that $d_2 \geq 4$. Because $d_2 \leq 2(p+1)=4$, $d_2=4$. Hence $X$ is $3$-connected by $\K$.
\end{proof}

\begin{them}\label{5fold} Let $X$ be a Fano $5$-fold of $\rho=1$. Then we have the following table: 

\begin{center}
\begin{tabular}{|c|c|c|c|c|}
\hline
 $p$ & $i_X$ & $X$ & $l_{\K}$ & $d_{\K}$ \\ \hline \hline
 $4$ & $6$ & $\PP^5$ & $1$ &  $1$   \\
 $3$ & $5$ & $\QQ^5$ & $2$ &  $1$ \\
 $2$ & $4$ & del Pezzo & $2$ &  $1$ \\
 $2$ & $2$ & coindex $4$ & $2$ &  $2$ \\
 $2$ & $1$ & coindex $5$ & $2$ &  $4$ \\
 $1$ & $3$ & Mukai & $3$ &  $1$ \\
 $1$ & $1$ & coindex $5$ & $3$ &  $3$ \\
 $0$ & $2$ & coindex $4$ & $5$ &  $1$ \\
 $0$ & $1$ & coindex $5$ & $5$ &  $2$ \\

   \hline
\end{tabular}

\end{center}

\end{them}

\begin{proof}
The computation of the length with respect to $\K$ is an immediate consequence of Theorem~\ref{CMSB}, Corollary~\ref{Mi2}, Lemma~\ref{0}, Theorem~\ref{length2} and Theorem~\ref{length3}. The other parts follow from the relation $p+2=i_Xd_{\K}$.
\end{proof}

\begin{conj} For a Fano $n$-fold of $\rho=1$, $(-K_X)^n \leq (n+1)^n$.
\end{conj}

This conjecture was proved by Hwang \cite{Hw3} in the case $n=4$. However, in general, this conjecture is still open for $n \geq 5$. For $n=5$, Hwang proved $(-K_X)^5$ is at most $81250$ \cite{Hw5}. To the best of my knowledge, it is the best known bound for $n=5$. As an application of Theorem~\ref{5fold}, we obtain an improved bound.   

\begin{cor}\label{canobound} For a Fano $5$-fold of $\rho=1$, $(-K_X)^5 \leq 9^5=59049$.\end{cor}

\begin{proof} If $p=0$, we know $(-K_X)^5 \leq 35310< 59049$ from \cite[Corollary~3]{Hw5}. When $p=4$ or $3$, our assertion is derived from Theorem~\ref{CMSB} and \ref{Mi2}. So it is enough to consider in the cases $p=2$ or $1$. From the definition of the locus ${\rm loc}^k(x)$, we know two general points can be joined by a chain of free rational curves whose anticanonical degree is at most $l_{\K}d_{\K}$. Furthermore a well-known argument \cite[Step~3]{KMM1} implies $(-K_X)^5 \leq (l_{\K}d_{\K})^5$ (see also the proof of Lemma~\ref{CC} (i) below). It follows from Theorem~\ref{5fold} that $l_{\K}d_{\K} \leq 9$ for $p=1,2$. So our assertion holds.
\end{proof}

\section{Lengths of Fano manifolds of coindex $3$}

In this section, we study Fano manifolds of coindex $3$. Because we already dealt with the case where $n:= \dim X \leq 5$ in Theorem~\ref{MT1}, we study the case where $n \geq 6$. 

\begin{pro}[\cite{KK}]\label{KK} Let $X$ be a projective variety and $\H \subset RatCurves^n(X)$ a proper dominating family of rational curves such that none of the associated curves has a cuspidal singularity. 
\begin{enumerate}
\item{For general $x \in X$, all curves in $\H$ passing through $x$ are smooth at $x$ and no two of them share a common tangent direction at $x$.  }
\item{Assume that for general $x \in X$ and any irreducible component $\H' \subset \H$, $\dim \H'_x \geq \frac{\dim X-1}{2}$ holds. Then $\H_x$ is irreducible. In particular, $\H$ is irreducible.}

\end{enumerate}
\end{pro}

\begin{lem}[{\cite[Corollary 1.4.3]{BS}}]\label{BS} Let $C$ be an integral curve and $L$ a spanned line bundle of degree $1$ on $C$. Then $(C, L) \cong (\PP^1, {\mathscr{O}}_{\PP^1}(1))$.

\end{lem}

\begin{nota} \rm We denote by $(d_1) \cap \dots \cap (d_k) \subset \PP^n$ a smooth complete intersection of hypersurfaces of degrees $d_1, \dots, d_k$, in particular, by $(d)^k$ if $d=d_1= \dots = d_k$. We denote by $G(k,n)$ a Grassmannian of $k$-planes in $\CC^n$, by $LG(k,n)$ a Lagrangian Grassmannian which is the variety of isotropic $k$-planes for a non-degenerate skew-symmetric bilinear form on $\CC^n$, by $S_k$ the spinor variety which is an irreducible component of the Fano variety of $k$-planes in $\QQ^{2k}$. 
We denote a simple exceptional linear algebraic group of Dynkin type $G$ simply by $G$ and for a dominant integral weight $\omega$ of $G$, the minimal closed orbit of $G$ in $\PP(V_{\omega})$ by $G(\omega)$, where $V_{\omega}$ is the irreducible representation space of $G$ with highest weight $\omega$. 
For example, $E_7({\omega}_1)$ is the minimal closed orbit of an algebraic group of type $E_7$ in $\PP(V_{{\omega}_1})$, where ${\omega}_1$ is the first fundamental dominant weight in the standard notation of Bourbaki \cite{Bo}.

\end{nota}

\begin{them}\label{coindex3} Let $X$ be a Fano manifold of coindex $3$ with ${\rm Pic}(X) \cong \ZZ[{\mathscr{O}}_X(1)]$ and $\K$ a minimal rational component of $X$. Assume that $n:= \dim X \geq 6$. Then $(l_{\K}, d_{\K})=(2, 1)$ except the case $X$ is a Lagrangian Grassmannian $LG(3,6)$. In the case $X=LG(3,6)$ we have $(l_{\K}, d_{\K})=(3, 1)$.
\end{them}

\begin{proof} We have an inequality $n+1 \geq p+2=(n-2)d_{\K}$. It follows from our assumption $n \geq 6$ that $(p,d_{\K})=(n-4,1)$.

 By Iskovskikh Theorem \cite{Is} or Mukai's classification result of Fano manifolds of coindex $3$ \cite{Muk,Me}, $X$ is  
\begin{enumerate} 
\item{a prime Fano manifold, which means ${\mathscr{O}}_X(1)$ is very ample, }
\item{a double cover $\pi: X \rightarrow \PP^n$ with a branch divisor $B \in |{\mathscr{O}}_{\PP^n}(6)|$, or}
\item{a double cover $\pi: X \rightarrow \QQ^n$ with a branch divisor $B \in |{\mathscr{O}}_{\QQ^n}(4)|$. }
\end{enumerate}

\begin{cl} For a general point $x \in X$, the variety of minimal rational tangents $\C_x \subset \PP(T_xX)$ is an equidimensional disjoint union of smooth projective varieties.
\end{cl}

When $X$ is prime, this follows from Proposition~\ref{hwlem1}. So we assume $X$ is as in $\rm (ii)$ or $\rm (iii)$. 
We denote by $Y$ the target of $\pi$ which is $\PP^n$ or $\QQ^n$. A Barth-type Theorem \cite{L} implies that ${\rm Pic}(X)\cong {\rm Pic}(Y)$ and that $\pi^*{\mathscr{O}}_Y(1) \cong {\mathscr{O}}_X(1)$, where ${\mathscr{O}}_Y(1)$ is the ample generator of the Picard group of $Y$.
By Proposition~\ref{stand} it is sufficient to show that the tangent map $\tau_x: \K_x \rightarrow \C_x$ is isomorphic. 

Since ${\mathscr{O}}_X(1)$ is spanned and $d_{\K}=1$, Lemma~\ref{BS} implies that any $l$ in $\K$ is isomorphic to $\PP^1$. Furthermore $\K$ is proper because $d_{\K}=1$ (see \cite[II. Proposition~2.14]{Ko}). So Proposition~\ref{KK} implies that $\tau_x$ is bijective. For $[l] \in \K_x$ we have ${\mathscr{O}}_X(1).l=1$. The projection formula shows that ${\mathscr{O}}_X(1).l=({\rm deg}\pi_l) \cdot {\mathscr{O}}_Y(1).\pi(l)$, where $\pi_l:l \rightarrow \pi(l)$ is the restriction of $\pi$ to $l$. Therefore $\pi(l) \subset Y$ is a standard line and $\pi_l$ is an isomorphism. Since $x \in X$ is general, we may assume that $l$ is free and the natural morphism between normal bundles $N_{l/X} \rightarrow N_{\pi(l)/Y}$ is generically surjective. Remark that $N_{\pi(l)/Y} = {\mathscr{O}}_{\PP^1}(1)^{\oplus n-1}$ or ${\mathscr{O}}_{\PP^1}(1)^{\oplus n-2} \oplus {\mathscr{O}}_{\PP^1}$. For $N_{l/X}= \oplus {\mathscr{O}}_{\PP^1}(a_i)$, if there exists $i$ such that $a_i \geq 2$, we see that the induced morphism ${\mathscr{O}}_{\PP^1}(a_i) \rightarrow N_{\pi(l)/Y}$ is a zero map. However this contradicts to the fact that $N_{l/X} \rightarrow N_{\pi(l)/Y}$ is generically surjective. This concludes that $l \subset X$ is a standard rational curve. Hence by Proposition~\ref{hwlem1}, $\tau_x$ is an immersion. As a consequence, we see $\tau_x$ is an embedding. So our claim holds.

Since $n \geq 6$, Proposition~\ref{hwlem1} implies that $\C_x \subset \PP(T_xX)$ is non-degenerate. When $n \geq 7$, we see $\C_x$ is irreducible. In fact, if there exist distinct irreducible components $\C_{x,1}, \C_{x,2}$ of ${\C_x}$, we see $\dim \C_{x,1} + \dim \C_{x,2} - \dim \PP(T_xX) \geq 0$. This implies that $\C_{x,1} \cap \C_{x,2} \neq \phi$. This contradicts the above claim. According to Zak's theorem on linear normality \cite{Za} and Theorem~\ref{HK}, we have $l_{\K}=2$. So it remains to prove the case $n=6$. If there exists an irreducible component of $\C_x$ whose secant variety coincides with $\PP(T_xX)$, we have $l_{\K}=2$. Therefore we assume that the secant variety of any irreducible component of $\C_x$ does not coincide with $\PP(T_xX)$. If $\C_x$ is irreducible, then it is the Veronese surface $v_2(\PP^2) \subset \PP^5$. This follows from Zak's classification of Severi varieties \cite{Za}. Here remark that the Veronese surface is the variety of minimal rational tangents of the Lagrangian Grassmannian $LG(3,6)$ at a general point (see \cite{HM,E,LM}). Thus in this case $X$ is isomorphic to $LG(3,6)$ by Theorem~\ref{HHM}. Because the secant variety of the Veronese surface is a hypersurface, it implies that $d_2=4$. Therefore we have $\l_{\K}=3$. If $\C_x$ is reducible, there exists disjoint irreducible components $V_1$ and $V_2$. Remark that we assumed that $S^1V_i$ does not coincide with $\PP(T_xX)$ for $i=1,2$. If $\dim S(V_1,V_2) \leq 4$, we have a point $q \in \PP^5 \setminus S(V_1,V_2) \cup S^1V_1 \cup S^1V_2$. For a projection $\pi_q$ from a point $q$, $\pi_q(V_i) \subset \PP^4$ is a surface. Hence it turns out that $\pi_q(V_1) \cap \pi_q(V_2) \subset \PP^4$ is non-empty. This contradicts $q \in S(V_1,V_2)$. Therefore we have $S(V_1,V_2)=\PP(T_xX)$. In particular, $S^1 \C_x=\PP(T_xX)$ and $l_{\K}=2$.

\end{proof}

\section{$2$-connectedness by lines, conic-connectedness and defectiveness of secant varieties.}

Here we remark a relation between $2$-connectedness by lines and conic-connectedness.

\begin{defi}[cf. {\cite{KaS,IR}}] \rm For a projective manifold $X \subset \PP^N$, we call $X$ {\it conic-connected} if there exists an irreducible conic passing through two general points on $X$.
\end{defi}

\begin{lem}[\cite{IR}]\label{CC} Let $X \subset \PP^N$ be a projective manifold which is covered by lines. Then
\begin{enumerate}
\item{if two general points on $X$ are connected by two lines, $X$ is conic-connected;}
\item{if $X$ is conic-connected, then the Fano index $i_X$ is at least $\frac{n+1}{2}$.}
\item{Assume that $X$ is conic-connected. Then two general points on $X$ are not connected by two lines if and only if $i_X = \frac{n+1}{2}$.}
\end{enumerate}
\end{lem}

\begin{proof} $\rm (i)$ is well known to the experts (cf. \cite{KMM1}, \cite[Proof of Proposition 5.8]{De}). Suppose that two general points $x_1, x_2 \in X$ are connected by two lines $l_1, l_2$. Then, without loss of generality, we may assume such two lines are free. By the gluing lemma, there exists a smoothing $(\pi:\C \rightarrow (T,0), F:\C \rightarrow X, s_1)$ of $l_1 \cup l_2 \subset X$ fixing $x_1$, where $s_1: T \rightarrow \C$ is a section of $\pi$ such that $s_1(0)=x_1 \in \pi^{-1}(0) \cong l_1 \cup l_2$ and $F \circ s_1 (T)= \{x_1\}$ (see \cite[Chapter II.7]{Ko}). According to a suitable base change, we may assume that there exists a section $s_2$ of $\pi$ such that $s_2(0)=x_2 \in \pi^{-1}(0) \cong l_1 \cup l_2$. Let $Z \subset X \times X$ be the set of points $(y_1,y_2) \in X \times X$ which can be joined by an irreducible conic in $X$. Then for a point $t \neq 0$ in $T$, $(s_1(t), s_2(t)) \in Z$. It turns out that $(x_1,x_2)$ is contained in the closure of $Z$. By the generality of $(x_1,x_2) \in X \times X$, we see $Z$ is dense in $X \times X$. Consequently our assertion holds. 

$\rm (ii)$ is in \cite{IR}. If $X$ is conic-connected, then there exists a smooth conic $C$ such that $T_X|_C$ is ample. This implies that $2i_X=\deg T_X|_C \geq n+1$. Hence $\rm (ii)$ holds. 

$\rm (iii)$ Suppose that $X$ is a conic-connected manifold which is not $2$-connected by lines. Then for general two points $x, y \in X$ there exists a smooth conic $f:\PP^1 \cong C \subset X$ passing through $x$ and $y$ such that $T_X|_C$ is ample. This implies that $H^1(\PP^1, f^*T_X(-2))=0$. Hence there is no obstruction in the deformation of $f$ fixing the marked points $x, y$. It turns out that 
\begin{equation}
\dim_{[f]}{\rm Hom}(\PP^1,X:f(0)=x, f(\infty)=y)=2i_X-n. 
\end{equation}
If $2i_X-n \geq 2$, Mori's Bend and Break implies $C$ degenerates into a union of two lines containing $x$ and $y$. This is a contradiction. Hence $2i_X-n \leq 1$. By combining $\rm (ii)$, we have $i_X = \frac{n+1}{2}$. Conversely if the Fano index satisfies $i_X=\frac{n+1}{2}$, it turns out from the same argument as in Theorem~\ref{HK} $\rm (i)$ that $X$ is not $2$-connected by lines. 
\end{proof}

\begin{exa} {\rm Let $S_4 \subset \PP^{15}$ be the $10$-dimensional spinor variety and let $X$ be $S_4$ or its linear section of dimension $n \geq 6$. Then $X$ is a Fano manifold of coindex $3$ with the genus $g:=\frac{H^n}{2}+1=7$, where $H$ is the ample generator of the Picard group of $S_4$. There exists a $6$-dimensional smooth quadric passing through two general points on $S_4$ \cite{ESB}. So $X$ is conic-connected and $2$-connected by lines. Hence two geneal points on $X$ can be connected by a chain of two lines which is obtained as a degeneration of a conic.} 
\end{exa}

\begin{exa} {\rm Let $X$ be a Grassmaniann $G(2,6) \subset \PP^{14}$ or its linear section of dimension $n \geq 6$. Then $X$ is a Fano manifold of coindex $3$ with the genus $g=8$. For two distinct points $x, y \in G(2,6)$, they correspond to $2$-dimensional vector subspaces $L_x, L_y \subset \CC^6$. Then there exists a $4$-dimensional vector subspace $V \subset \CC^6$ which contains the join $<L_x, L_y>$. This implies that $x, y$ is contained in a $4$-dimensional quadric $Q^4 \cong G(2,4) \subset G(2,6)$. So $X$ is conic-connected and $2$-connected by lines.} 

\end{exa}

\begin{rem} \rm $X:= G(2,6) \cap (1)^3 \subset \PP^{14}$ is a $5$-dimensional Fano manifold of index $3$. According to Theorem~\ref{5fold}, $X$ is $3$-connected by lines. However $X$ is conic-connected. This example shows that our chain of minimal rational curves connecting two general points is not necessary a chain with minimal total degree.
\end{rem}

\begin{them}\label{2/3} Let $X$ be a prime Fano $n$-fold of $\rho=1$ with $i_X=\frac{2}{3}n$ and $\K$ a minimal rational component of $X$. Then $l_{\K}=2$ except the following cases: 
\begin{enumerate}
\item{$(3) \subset \PP^4$ a hypersurface of degree $3$.}
\item{$(2) \cap (2) \subset \PP^5$ a complete intersection of two hyperquadrics.}
\item{$G(2,5) \cap (1)^3 \subset \PP^6$ a $3$-dimensional linear section of $G(2,5)$.}
\item{$LG(3,6)$ a Lagrangian Grassmannian.}
\item{$G(3,6)$ a Grassmannian.}
\item{$S_5$ a spinor variety.}
\item{$E_7({\omega}_7)$ a rational homogeneous manifold of type $E_7$.}
\end{enumerate}

Furthermore in the cases $\rm (i)-(vii)$ we have $l_{\K}=3$.
\end{them}

\begin{proof}
According to the assumption that $3i_X=2n$, $n$ is $3$, $6$, or at least $9$. If $n=3$, $X$ is a del Pezzo $3$-fold. So Remark~\ref{3fold} implies that $(l_{\K}, d_{\K})=(3,1)$. Hence $X$ is isomorphic to one of the manifolds listed in $\rm (i), (ii)$ or $\rm (iii)$ by the Fujita-Iskovskikh's classification result \cite{F1,F2,F3}. In the case where $n=6$, we have $\l_{\K}=2$ or $X$ is $LG(3,6)$ by Theorem~\ref{coindex3}.

From here, we make the assumption $n \geq 9$. In this case, we have $2i_X > n+1$. So $d_{\K}=1$, that is, $X$ is covered by lines. By Proposition~\ref{hwlem1} the variety of minimal rational tangents ${\C}_x \subset \PP(T_xX)$ is smooth irreducible and non-degenerate. It follows from our assumption that $2p \geq n-1$. Hence ${\C}_x \subset \PP(T_xX)$ is a non-degenerate irreducible projective manifold of dimension $\frac{2}{3}n-2$. By Zak's theorem on linear normality, a classification of Severi varieties \cite{Za} and the assumption that $n \geq 9$, $S^1{\C}_x = \PP(T_xX)$ or ${\C}_x \subset \PP(T_xX)$ is isomorphic to the Segre product $\PP^2 \times \PP^2 \subset \PP^8$, the Grassmann variety $G(2,6) \subset \PP^{14}$ or $E_6$-variety $E_6({\omega}_1) \subset \PP^{26}$. In the former case Theorem~\ref{HK} implies that $l_{\K}=2$. So we assume the latter holds. Remark that the above Segre variety, Grassmann variety and $E_6$-variety are varieties of minimal rational tangents of $G(3,6)$, $S_5$ and $E_7({\omega}_7)$ respectively (For example, see \cite{HM,E,LM}). By Theorem~\ref{HHM}, $X$ is isomorphic to one of these varieties. In this case,  since ${\C}_x \subset \PP(T_xX)$ is a Severi variety, $S^1{\C}_x \subset \PP(T_xX)$ is a hypersurface \cite{Za}. This implies $d_2=n-1$ \cite[Theorem 3.14]{HK}. Hence $l_{\K}=3$. 

\end{proof}

\begin{cor}\label{Zak1} Let $X$ be a prime Fano $n$-fold of $\rho=1$ with $i_X=\frac{2}{3}n$ and $\K$ a minimal rational component of $X$. Assume that $n \geq 6$. Then the following are equivalent.
\begin{enumerate}
\item{$l_{\K} \neq 2$.}
\item{$l_{\K}=3$.}
\item{$X \subset \PP(H^0(X, {\mathscr{O}}_X(1))^{\vee})$ is not conic-connected.}
\item{The dimension of the secant variety $S^1X \subset \PP(H^0(X, {\mathscr{O}}_X(1))^{\vee})$ is $2n+1$.}
\item{The variety of minimal rational tangents at a general point ${\C}_x \subset \PP(T_xX)$ is a Severi variety.}
\item{$X \subset \PP(H^0(X, {\mathscr{O}}_X(1))^{\vee})$ is projectively equivalent to one of the manifolds listed in Theorem~\ref{2/3} $\rm (iv)-(vii)$.}
\end{enumerate}

\end{cor}

\begin{proof} By the above theorem and its proof, $\rm (i), (ii), (v)$ and $\rm (vi)$ are equivalent to each other. In general, if $X \subset \PP^N$ is conic-connected, then the dimension of the secant variety $S^1X$ is smaller than $2n +1$ (see \cite[Proposition 3.2]{IR1} and \cite[Theorem 2.1]{R}). Hence $\rm (iv) \Rightarrow (iii)$ holds. $\rm (iii) \Rightarrow (i)$ follows from Lemma~\ref{CC}. Finally, $\rm (vi) \Rightarrow (iv)$ comes from \cite{K}.
\end{proof}

\begin{rem}\label{conicdefective} \rm 
Corollary~\ref{Zak1} and Theorem~\ref{bound} implies that $i_X=\frac{2}{3}n$ is also a boundary of conic-connectedness and defectiveness of the secant variety (Remark~\ref{defective}): 

\begin{center}
\begin{tabular}{|c|c|c|c|}
\hline
Property & $i_X>\frac{2}{3}n$ & \multicolumn{2}{|c|}{$i_X=\frac{2}{3}n$} \\ \hline \hline
 $l_{\K}$ & $2$ & $2$ & $3$   \\ 
 Conic-connectedness & Yes & Yes & No  \\ 
 Defectiveness of the secant variety & Yes & Yes & No  \\ 
 \hline
\end{tabular}
\end{center}

\end{rem}

\section{Lengths of Fano manifolds admitting the structures of double covers}

Let $X$ be a Fano $n$-fold with Pic$(X) \cong \ZZ[{\mathscr{O}}_X(1)]$, where ${\mathscr{O}}_X(1)$ is ample and $n:=\dim X \geq 3$. In this section, we assume that $X$ is a double cover of a projective manifold $\pi: X \rightarrow Y$. A Barth-type Theorem \cite{L} implies that Pic$(X)\cong$ Pic$(Y)$ and that $\pi^*{\mathscr{O}}_Y(1) \cong {\mathscr{O}}_X(1)$, where ${\mathscr{O}}_Y(1)$ is the ample generator of the Picard group of $Y$. It follows from the ramification formula of the branched cover that $Y$ is a Fano manifold. We denote by $B \in |{\mathscr{O}}_Y(b)|$ the branch divisor of $\pi$ and by $\R_1$ the family of rational curves of degree $1$ ${\rm RatCurves}^n_1(X)$. We assume that $\R_1$ is a dominating family. Then we can define the $k$-th locus ${\rm loc}^k_{\R_1}(x)$ and the length with respect to $\R_1$ as in Definition~\ref{loc} and Definition~\ref{leng}.

\begin{pro}\label{criterion} Let $X$ and $\R_1$ be as above. Then the following holds.
\begin{enumerate}
\item{$X$ is $2$-connected by $\R_1$ if and only if for general $x_1, x_2 \in X$, $\pi({\rm loc}^1_{\R_1}(x_1)) \cap \pi({\rm loc}^1_{\R_1}(x_2)) \neq \phi$.}
\item{Under the assumption that ${\mathscr{O}}_Y(1)$ is spanned, $X$ is $2$-connected by $\R_1$ if and only if for general points $y_1, y_2 \in Y$ there exist curves $l_1 \ni y_1, l_2 \ni y_2$ on $Y$ such that $l_1 \cap l_2 \neq \phi$, ${\mathscr{O}}_Y(1).l_i=1$ and ${\rm length}_q(B \cap l_i) \equiv 0$ {\rm mod} $2$ for any $q \in Y$ and $i=1,2$.}
\end{enumerate}
\end{pro}

\begin{proof} $\rm (i)$ The "only if" part is trivial. We show the converse. Let $x_1, x_2$  be points on $X$ which are not contained in the ramification locus of $\pi$ and set $y_2:=\pi(x_2)$. Then we have ${\pi}^{-1}(y_2)= \{x_2, {x_2}'\}$. We assume there exists a point $z \in \pi({\rm loc}^1_{\R_1}(x_1)) \cap \pi({\rm loc}^1_{\R_1}(x_2))$. Then there exists a curve $[l_{x_i}] \in \R_1$ such that $x_i \in l_{x_i}$ and $z \in \pi(l_{x_i})$ for $i=1,2$. Since $\pi(l_{x_2}) \subset Y$ is a curve of degree $1$, ${\pi}^{-1}(\pi(l_{x_2}))$ is a curve of degree $2$. It follows from the inclusion $l_{x_i} \subset {\pi}^{-1}(\pi(l_{x_i}))$ that there exists a curve $[l_{{x_2}'}] \in \R_{1, {x_2}'}$ such that ${\pi}^{-1}(\pi(l_{x_2}))=l_{x_2} \cup l_{{x_2}'}$. Our assumption implies that $l_{x_1} \cap l_{x_2} \neq \phi$ or $l_{x_1} \cap l_{{x_2}'} \neq \phi$. So $x_2$ or ${x_2}'$ is contained in ${\rm loc}^2_{\R_1}(x_1)$. This means $\pi|_{{\rm loc}^2_{\R_1}(x_1)}: {\rm loc}^2_{\R_1}(x_1) \rightarrow Y$ is dominant. Since $\pi|_{{\rm loc}^2_{\R_1}(x_1)}$ is proper, it is surjective. Hence we see $X = {\rm loc}^2_{\R_1}(x_1)$.

$\rm (ii)$ Suppose that ${\mathscr{O}}_Y(1)$ is spanned. Let $l$ be a rational curve on $Y$ satisfying ${\mathscr{O}}_Y(1).l=1$. $\pi^{-1}(l)$ is denoted by $C$.  From $\rm (i)$, it is sufficient to show the following claim.

\begin{cl}
$C$ is reducible if and only if ${\rm length}_q(B \cap l) \equiv 0$ ${\rm mod}$ $2$ for any $q \in Y$. 
\end{cl}

For the double cover $\pi:X \rightarrow Y$, we have $\pi_*{\mathscr{O}}_X \cong {\mathscr{O}}_Y \oplus L^{-1}$, where $L$ is an ample line bundle on $Y$ which satisfies $L^{\otimes 2} \cong {\mathscr{O}}_Y(B)$. Furthermore there exists a morphism $X \hookrightarrow \LL :={\rm Spec}({\rm Sym} L^{-1})$ over $Y$. Denote by $\pi_C$ the restriction of $\pi$ to $C$ and by $L_l$ one of $L$ to $l$. Since $X$ is a divisor on $\LL$, we can obtain the defining equation of $X$ on $\LL$. 
In particular, we see that there exists a global section $s \in \Gamma(C,{\pi_C}^*L_l)$ such that $s^2={\pi_C}^*{\phi}$, where $\phi \in \Gamma(\PP^1,{\mathscr{O}}_{\PP^1}(b))$ and $(\phi=0)= l \cap B$ as divisors of $l$. We may assume that $\pi_C$ is unramified at ${\infty} \in \PP^1$. Then we see $C$ is reducible if and only if $\pi_C^{-1}(\AA^1)$ is reducible. Without loss of generality, we may assume that $\phi=(x-a_1y)\cdots(x-a_by)$, where $a_i \in \CC$ and $\Gamma(\PP^1,{\mathscr{O}}_{\PP^1}(b)) \cong \bigoplus_{i=0}^{b} \CC x^iy^{b-i}$. Furthermore we may assume $\pi_C^{-1}(\AA^1)=(s^2=(x-a_1)\cdots(x-a_b)) \subset \AA^2$. Thus $C$ is reducible if and only if the cardinality $\# \{ j|a_j=a_i \} \equiv 0$ mod $2$ for any $i$. Hence we obtain our assertion.
\end{proof}

\begin{cor}\label{f.proj} Let $X$, $Y$ and $\R_1$ be as above. If $Y=\PP^n$ and $n \geq b$, then $X$ is $2$-connected by $\R_1$.
\end{cor}

\begin{proof} There exists a standard rational curve $f: \PP^1 \rightarrow X$ such that $f^{*}T_X \cong {\mathscr{O}}(2) \oplus {\mathscr{O}}(1)^p \oplus {\mathscr{O}}^{n-1-p}$. By the ramification formula, the Fano index $i_X$ of $X$ is equal to $n+1-\frac{b}{2}$. It follows from the assumption $n \geq b$ that $i_X > \frac{n+1}{2}$. Hence we have deg $f_*(\PP^1)=1$ and $p=n-\frac{b}{2}-1$. For general two points $x_1, x_2 \in X$, 
\begin{equation}
\dim \pi({\rm loc}^1_{\R_1}(x_1)) + \dim \pi({\rm loc}^1_{\R_1}(x_2)) - \dim \PP^n=2(n-\frac{b}{2})-n=n-b \geq 0. 
\end{equation}
Hence Proposition~\ref{criterion} implies that $X$ is $2$-connected by $\R_1$.

\end{proof}

\begin{cor}\label{h.proj} Let $X$, $Y$ and $\R_1$ be as in above. If $Y \subset \PP^{n+1}$ is a hypersurface of degree $d$ and $n \geq 2d+b-1$, then $X$ is $2$-connected by $\R_1$.
\end{cor}

\begin{proof} There exists a standard rational curve $f: \PP^1 \rightarrow X$ such that $f^{*}T_X \cong {\mathscr{O}}(2) \oplus {\mathscr{O}}(1)^p \oplus {\mathscr{O}}^{n-1-p}$. By the ramification formula, the Fano index $i_X$ of $X$ is equal to $n+2-d-\frac{b}{2}$. Since we have $i_X > \frac{n+1}{2}$, it implies that deg $f_*(\PP^1)=1$ and $p=n-\frac{b}{2}-d$. For general two points $x_1, x_2 \in X$, 
\begin{equation}
\dim \pi({\rm loc}^1_{\R_1}(x_1)) + \dim \pi({\rm loc}^1_{\R_1}(x_2)) - \dim \PP^{n+1} \geq 0. 
\end{equation}
Thus $X$ is $2$-connected by $\R_1$.

\end{proof}

{\bf Acknowledgements}
The author would like to thank his supervisor Professor Hajime Kaji for fruitful advice and encouragement. He would also like to thank the referee for the careful reading of the manuscript and useful comments. The author is supported by Research Fellowships of the Japan Society for the Promotion of Science for Young Scientists.

\end{document}